\documentclass[reqno]{amsart}

\usepackage{amssymb}
\usepackage{graphicx}
\usepackage{amscd}
\usepackage[hidelinks]{hyperref}
\usepackage{color}
\usepackage{float}
\usepackage{graphics,amsmath,amssymb}
\usepackage{amsthm}
\usepackage{amsfonts}
\usepackage{latexsym}
\usepackage{epsf}
\usepackage{xifthen}
\usepackage{mathrsfs}
\usepackage{dsfont}
\usepackage{makecell}
\usepackage[FIGTOPCAP]{subfigure}
\usepackage{amsmath}
\allowdisplaybreaks[4]
\usepackage{listings}
\usepackage{etoolbox}
\usepackage{fancyhdr}
\usepackage{pdflscape}
\usepackage[title,toc,titletoc]{appendix}
\usepackage{enumitem}
\usepackage[noadjust]{cite}
\usepackage{tikz}
\usepackage{ytableau}

 \newtheoremstyle{mytheorem}
 {3pt}
 {3pt}
 {\slshape}
 {}
 {\bfseries}
 {.}
 { }
 {}

\numberwithin{equation}{section}

\theoremstyle{theorem}
\newtheorem{theorem}{Theorem}[section]
\newtheorem*{theorem*}{Theorem}

\newtheorem{proposition}[theorem]{Proposition}

\theoremstyle{definition}
\newtheorem{definition}{Definition}[section]

\newtheorem{example}{Example}[section]
\newtheorem*{example*}{Example}

\theoremstyle{remark}

\newtheorem*{remark*}{Remark}
\newtheorem*{remarks*}{Remarks}

\newcommand{\Keywords}[1]{\ifthenelse{\isempty{#1}}{}{\smallskip \smallskip \noindent \textbf{Keywords}. #1}}
\newcommand{\MSC}[2][2010]{\ifthenelse{\isempty{#2}}{}{\smallskip \smallskip \noindent \textbf{#1MSC}. #2}}
\newcommand{\abstractnote}[1]{\ifthenelse{\isempty{#1}}{}{\smallskip \smallskip \noindent \textsuperscript{\dag}#1}}

\makeatletter
\def\specialsection{\@startsection{section}{1}%
  \z@{\linespacing\@plus\linespacing}{.5\linespacing}%
  {\normalfont}}
\def\section{\@startsection{section}{1}%
  \z@{.7\linespacing\@plus\linespacing}{.5\linespacing}%
  {\normalfont\scshape}}
\patchcmd{\@settitle}{\uppercasenonmath\@title}{\Large\boldmath}{}{}
\patchcmd{\@settitle}{\begin{center}}{\begin{flushleft}}{}{}
\patchcmd{\@settitle}{\end{center}}{\end{flushleft}}{}{}
\patchcmd{\@setauthors}{\MakeUppercase}{\normalsize}{}{}
\patchcmd{\@setauthors}{\centering}{\raggedright}{}{}
\patchcmd{\section}{\scshape}{\large\bfseries\boldmath}{}{}
\patchcmd{\subsection}{\bfseries}{\bfseries\boldmath}{}{}
\renewcommand{\@secnumfont}{\bfseries}
\patchcmd{\@startsection}{\@afterindenttrue}{\@afterindentfalse}{}{}
\patchcmd{\abstract}{\leftmargin3pc}{\leftmargin1pc}{}{}

\def\maketitle{\par
  \@topnum\z@ 
  \@setcopyright
  \thispagestyle{empty}
  \ifx\@empty\shortauthors \let\shortauthors\shorttitle
  \else \andify\shortauthors
  \fi
  \@maketitle@hook
  \begingroup
  \@maketitle
  \toks@\@xp{\shortauthors}\@temptokena\@xp{\shorttitle}%
  \toks4{\def\\{ \ignorespaces}}
  \edef\@tempa{%
    \@nx\markboth{\the\toks4
      \@nx\MakeUppercase{\the\toks@}}{\the\@temptokena}}%
  \@tempa
  \endgroup
  \c@footnote\z@
  \@cleartopmattertags
}
\makeatother

\newcommand{\sP}{\mathscr{P}}
\newcommand{\sC}{\mathscr{C}}

\newcommand{\qbinom}[2]{\begin{bmatrix}#1\\#2\end{bmatrix}}

\title{Reciprocity between partitions and compositions}

\author[G. Beck]{George Beck}
\address[G. Beck]{Department of Mathematics and Statistics, Dalhousie University, Halifax, Nova Scotia, B3H 4R2, Canada\newline\indent and Wolfram Research, Inc., Champaign, IL 61820, USA}
\email{george.beck@gmail.com}

\author[S. Chern]{Shane Chern}
\address[S. Chern]{Department of Mathematics and Statistics, Dalhousie University, Halifax, Nova Scotia, B3H 4R2, Canada}
\email{chenxiaohang92@gmail.com}

\begin{document}

\maketitle

\begin{abstract}

In this paper, we extend the work of Andrews, Beck and Hopkins by considering partitions and compositions with bounded gaps between each pair of consecutive parts. We show that both their generating functions and two matrices determined by them satisfy certain reciprocal relations.


\MSC{05A17, 11P84.}
\end{abstract}

\section{Introduction}

A \textit{partition} of a natural number $n$ is a \textit{nondecreasing} sequence of positive integers whose sum equals $n$. For any partition $\lambda$, we denote by $|\lambda|$ the \textit{size} (that is, the sum of the parts) of $\lambda$, and by $\ell(\lambda)$ the \textit{length} (that is, the number of parts) of $\lambda$. We write $\lambda$ as $(\lambda_1,\lambda_2,\ldots,\lambda_{\ell(\lambda)})$ with $\lambda_1\le \lambda_2\le \cdots\le \lambda_{\ell(\lambda)}$.

In answer to a conjecture of Hanna \cite[A126796]{OEIS}, which is related to complete partitions defined by Park \cite{Par1998}, Andrews, Beck and Hopkins \cite{ABH2020} introduced $m$-step partitions.

\begin{definition}
	Given a nonnegative integer $m$, a partition $(\lambda_1,\lambda_2,\ldots,\lambda_{\ell(\lambda)})$ is called \textit{$m$-step} if $\lambda_j\le m+\sum_{i=0}^{j-1}\lambda_i$ for each part $\lambda_j$, with $\lambda_0=0$ for convenience.
\end{definition}

Let $K(n,m)$ denote the number of $m$-step partitions of $n$. Andrews, Beck and Hopkins extended Hanna's conjecture and proved the following relation.

\begin{theorem}[{\cite[Theorem 9]{ABH2020}}]\label{th:ABH1}
	For each positive integer $m$,
	\begin{equation}
	\sum_{n\ge 0}K(n,m)q^n (q;q)_{m+n}=1.
	\end{equation}
\end{theorem}

This relation then leads to a more surprising result for two matrices determined by partitions. First, let $M(n,m)$ be defined by
\begin{equation}
\sum_{m,n\ge 1}M(n,m)x^m q^n = \sum_{\lambda} (-1)^{\ell(\lambda)+1}x^{\lambda_{\ell(\lambda)}}q^{|\lambda|},
\end{equation}
where the summation is over nonempty partitions into distinct parts; this is motivated by Schneider's partition-theoretic M{\"o}bius function \cite[Definition 3.1]{Sch2017}. Andrews, Beck and Hopkins then introduced two infinite lower-triangular matrices:
\begin{enumerate}[label={\textup{(\roman*).}},leftmargin=*,labelsep=0cm,align=left]
	\item $\mu$ is defined by $(\mu)_{i,j\ge 1}:=M(i+1,j+1)$;
	\item $\gamma$ is defined by $(\gamma)_{i,j\ge 1}:=K(i-j,j)$.
\end{enumerate}

\begin{theorem}[{\cite[Theorem 10]{ABH2020}}]\label{th:ABH2}
	$\mu$ and $\gamma$ are inverses of one another.
\end{theorem}

Observing that partitions into distinct parts are partitions where the gap between each pair of consecutive parts is at least $1$, we extend the partition function $M$ to count partitions where the gap between each pair of consecutive parts is at least $g$.

\begin{definition}\label{def:F}
	Let $g$ be a nonnegative integer. Let $\sP_g$ denote the set of partitions $\lambda=(\lambda_1,\lambda_2,\ldots,\lambda_{\ell(\lambda)})$ such that $\lambda_{i+1}-\lambda_{i}\ge g$ for all $1\le i\le \ell(\lambda)-1$.
\end{definition}

We define $M_g(n,m)$ by
\begin{equation}
\sum_{m,n\ge 1}M_g(n,m)x^m q^n = \sum_{\lambda\in\sP_g\,(\lambda\ne\emptyset)} (-1)^{\ell(\lambda)+1}x^{\lambda_{\ell(\lambda)}}q^{|\lambda|}.
\end{equation}

For the counterpart of the partition function $K$, we need to go beyond partitions to compositions.

A \textit{composition} of a natural number $n$ is a sequence of positive integers whose sum equals $n$. For any composition $\kappa$, we denote by $|\kappa|$ the \textit{size} of $\kappa$, and by $\ell(\kappa)$ the \textit{length} of $\kappa$. We write $\kappa$ as $(\kappa_1,\kappa_2,\ldots,\kappa_{\ell(\kappa)})$.

As with partitions, we consider compositions with bounded gaps between consecutive parts.

\begin{definition}\label{def:G}
	Let $g$ be a nonnegative integer. Let $\sC_g$ denote the set of compositions $\kappa=(\kappa_1,\kappa_2,\ldots,\kappa_{\ell(\kappa)})$ such that $\kappa_{i+1}-\kappa_{i}\ge -(g-1)$ for all $1\le i\le \ell(\kappa)-1$.
\end{definition}

The following fascinating result is due to Andrews \cite{And1981}. The case $g=2$ was also considered by Jovovic and Zeilberger; the involution of the latter is generalized in the proof of Theorem \ref{th:G}. See \cite[A003116]{OEIS}.
\begin{theorem}[{\cite[Theorem 1]{And1981}}]
	For each nonnegative integer $g$,
	\begin{equation}\label{eq:G-toy}
	\sum_{\kappa\in\sC_g}x^{\ell(\kappa)}q^{|\kappa|} = \dfrac{1}{\displaystyle  \sum_{\ell\ge 0}\frac{(-x)^\ell q^{\ell+\binom{\ell}{2}g}}{(q;q)_\ell}}.
	\end{equation}
\end{theorem}

Notice also that the summation in the denominator of \eqref{eq:G-toy} is related to partitions with gaps at least $g$:
\begin{equation}\label{eq:F-toy}
\sum_{\lambda\in\sP_g}(-x)^{\ell(\lambda)}q^{|\lambda|} = \sum_{\ell\ge 0}\frac{(-x)^\ell q^{\ell+\binom{\ell}{2}g}}{(q;q)_\ell}.
\end{equation}
The above generating function identities are special cases in a broader setting. This will be discussed in Section \ref{sec:gf}.

Returning to the counterpart of the partition function $K$, let us generalize $m$-step partitions to $m$-step compositions.

\begin{definition}
	Given a nonnegative integer $m$, a composition $(\kappa_1,\kappa_2,\ldots,\kappa_{\ell(\kappa)})$ is called \textit{$m$-step} if $\kappa_j\le m+\sum_{i=0}^{j-1}\kappa_i$ for each part $\kappa_j$, with $\kappa_0=0$ for convenience.
\end{definition}

Let $K_g(n,m)$ denote the number of $m$-step compositions of $n$ in $\sC_g$.

As we have seen, $M(n,m)=M_1(n,m)$. To see that $K(n,m)=K_1(n,m)$, we notice that $\sC_1$ is the set of compositions $\kappa=(\kappa_1,\kappa_2,\ldots,\kappa_{\ell(\kappa)})$ such that $\kappa_{i+1}-\kappa_{i}\ge 0$ for all $1\le i\le \ell(\kappa)-1$. So $\sC_1$ coincides with the set of partitions.

Now we extend Theorem \ref{th:ABH2}, which is the main result of this paper.

\begin{theorem}\label{th:mu-gamma}
	For each positive integer $g$, let $\mu_g$ be the infinite lower-triangular matrix defined by $(\mu_g)_{i,j\ge 1}:=M_g(i+g,j+g)$ and let $\gamma_g$ be the infinite lower-triangular matrix defined by $(\gamma_g)_{i,j\ge 1}:=K_g(i-j,j)$. Then $\mu_g$ and $\gamma_g$ are inverses of one another.
\end{theorem}

The case $g=2$ is presented in Example \ref{ex:g=2}. This result is the $s=1$ case of Theorem \ref{th:mu-gamma-ext}, whose proof is established in Section \ref{sec:mat} along with a generalization of Theorem \ref{th:ABH1}.

Throughout, we adopt the usual notation: for $n\in\mathbb{N}\cup\{\infty\}$, the \textit{$q$-Pochhammer symbol} is defined by
\begin{align*}
(a;q)_n:=\prod_{k=0}^{n-1}(1-a q^k).
\end{align*}
Also, the \textit{$q$-binomial coefficient} is defined by
\begin{align*}
\qbinom{A}{B}_q:=\begin{cases}
\dfrac{(q;q)_A}{(q;q)_B(q;q)_{A-B}} & \text{if $0\le B\le A$},\\[12pt]
0 & \text{otherwise}.
\end{cases}
\end{align*}
Further, for any formal power series $f(q)$, we denote by $[q^n]f(q)$ the coefficient of $q^n$ in its expansion.

\section{Generating functions}\label{sec:gf}

We first refine the sets $\sP_g$ and $\sC_g$. Throughout, let $s$ be a positive integer.

\begin{definition}\label{def:refine}
	Let $\sP_g^{(s)}$ denote the set of partitions in $\sP_g$ such that all parts are at least $s$. Let $\sC_g^{(s)}$ denote the set of compositions in $\sC_g$ such that all parts are at least $s$.
\end{definition}

\begin{proposition}\label{prop:F}
	Let $g$ and $m$ be nonnegative integers and $s$ be a positive integer. Let $\sP_g^{(s)}$ be as in Definition \ref{def:refine} and let $\sP_{g,\le m}^{(s)}$ denote the set of partitions in $\sP_g^{(s)}$ such that $\lambda_{\ell(\lambda)}\le m$. Let
	$$P_g^{(s)}(x,q):=\sum_{\lambda\in\sP_g^{(s)}}x^{\ell(\lambda)}q^{|\lambda|}$$
	and
	$$P_{g,\le m}^{(s)}(x,q):=\sum_{\lambda\in\sP_{g,\le m}^{(s)}}x^{\ell(\lambda)}q^{|\lambda|}.$$
	Then
	\begin{equation}\label{eq:F}
	P_g^{(s)}(x,q)=\sum_{\ell\ge 0}\frac{x^\ell q^{\ell s+\binom{\ell}{2}g}}{(q;q)_\ell}
	\end{equation}
	and
	\begin{equation}\label{eq:Fm}
	P_{g,\le m}^{(s)}(x,q)=1+\sum_{\ell\ge 1}x^\ell q^{\ell s+\binom{\ell}{2}g}\qbinom{m-s+1-(\ell-1)(g-1)}{\ell}_q.
	\end{equation}
\end{proposition}

\begin{proof}
	We prove \eqref{eq:Fm}; then \eqref{eq:F} follows by letting $m\to\infty$ in \eqref{eq:Fm}. For \eqref{eq:Fm}, we notice that if $(\lambda_1,\lambda_2,\ldots,\lambda_{\ell(\lambda)})$ is in $\sP_{g,\le m}^{(s)}$, then $(\lambda_1-s,\lambda_2-s-g,\ldots, \lambda_{\ell(\lambda)}-s-(\ell-1)g)$ gives a partition with at most $\ell(\lambda)$ parts and largest part not exceeding $m-s-(\ell-1)g$.
\end{proof}

\begin{theorem}\label{th:G}
	Let $g$ be a nonnegative integer and $s$ be a positive integer. Let $\sC_g^{(s)}$ be as in Definition \ref{def:refine} and let $\sC_{g,\ge m}^{(s)}$ denote the set of compositions in $\sC_g^{(s)}$ such that $\kappa_{1}\ge m$. Let
	$$C_g^{(s)}(x,q):=\sum_{\kappa\in\sC_g^{(s)}}x^{\ell(\kappa)}q^{|\kappa|}$$
	and
	$$C_{g,\ge m}^{(s)}(x,q):=\sum_{\kappa\in\sC_{g,\ge m}^{(s)}}x^{\ell(\kappa)}q^{|\kappa|}.$$
	Then for each positive integer $m$,
	\begin{equation}\label{eq:Gm}
	C_{g,\ge m}^{(s)}(x,q)=\frac{P_{g,\le {m-1}}^{(s)}(-x,q)}{P_g^{(s)}(-x,q)}
	\end{equation}
	and in particular,
	\begin{equation}\label{eq:G}
	C_g^{(s)}(x,q)=\frac{1}{P_g^{(s)}(-x,q)}.
	\end{equation}
\end{theorem}

\begin{proof}
	We first introduce a set of pairs $\Pi=\Pi^{(s)}:=\{(\lambda,\kappa)\}$ where
	\begin{align*}
	\left\{
	\begin{aligned}
	\lambda&=(\lambda_1,\lambda_2,\ldots,\lambda_{\ell(\lambda)})\in\sP_g^{(s)},\\
	\kappa&=(\kappa_1,\kappa_2,\ldots,\kappa_{\ell(\kappa)})\in\sC_g^{(s)}.
	\end{aligned}
	\right.
	\end{align*}
	Then we have an involution $\phi$ on $\Pi$ given by $\phi(\lambda,\kappa)\mapsto (\phi(\lambda),\phi(\kappa))$ where if $\lambda$ and $\kappa$ are nonempty and
	\begin{enumerate}[label={\textup{(\roman*).}},leftmargin=*,labelsep=0cm,align=left]
		\item $\kappa_{\ell(\kappa)}-\lambda_{\ell(\lambda)}\ge g$,
		\begin{equation*}
		\left\{
		\begin{aligned}
		\phi(\lambda) &= (\lambda_1,\lambda_2,\ldots,\lambda_{\ell(\lambda)},\kappa_{\ell(\kappa)}),\\
		\phi(\kappa) &= (\kappa_1,\kappa_2,\ldots,\kappa_{\ell(\kappa)-1});
		\end{aligned}
		\right.
		\end{equation*}
		
		\item $\kappa_{\ell(\kappa)}-\lambda_{\ell(\lambda)}\le g-1$,
		\begin{equation*}
		\left\{
		\begin{aligned}
		\phi(\lambda) &= (\lambda_1,\lambda_2,\ldots,\lambda_{\ell(\lambda)-1}),\\
		\phi(\kappa) &= (\kappa_1,\kappa_2,\ldots,\kappa_{\ell(\kappa)},\lambda_{\ell(\lambda)});
		\end{aligned}
		\right.
		\end{equation*}
	\end{enumerate}
	and
	\begin{align*}
	\phi(\emptyset,\emptyset)&=(\emptyset,\emptyset),\\
	\phi((\lambda_1,\lambda_2,\ldots,\lambda_{\ell(\lambda)}),\emptyset)&=((\lambda_1,\lambda_2,\ldots,\lambda_{\ell(\lambda)-1}),(\lambda_{\ell(\lambda)})),\\
	\phi(\emptyset,(\kappa_1,\kappa_2,\ldots,\kappa_{\ell(\kappa)}))&=((\kappa_{\ell(\kappa)}),(\kappa_1,\kappa_2,\ldots,\kappa_{\ell(\kappa)-1})).
	\end{align*}
	It is straightforward to verify that for any $(\lambda,\kappa)\in \Pi$, we have $\phi(\lambda,\kappa)\in \Pi$, $\phi\circ \phi(\lambda,\kappa) = (\lambda,\kappa)$, $|\lambda|+|\kappa|=|\phi(\lambda)|+|\phi(\kappa)|$ and $\ell(\lambda)+\ell(\kappa)=\ell(\phi(\lambda))+\ell(\phi(\kappa))$. Further, if we assign a weight $w$ to each $(\lambda,\kappa)$ by $w(\lambda,\kappa) = (-1)^{\ell(\lambda)}$, then $w(\phi(\lambda),\phi(\kappa))=-w(\lambda,\kappa)$ with the only exception that $w(\phi(\emptyset),\phi(\emptyset))=w(\emptyset,\emptyset)=1$.
	
	Now, for $m\ge 1$, we denote by $\Pi_{m}$ the subset of $\Pi$ such that $\kappa_1\ge m$ or $\kappa=\emptyset$. If we further denote by $\Pi_m^*$ the subset of $\Pi_{m}$ such that $(\lambda,\kappa)\in \Pi_m^*$ is not of the form $(\lambda,\emptyset)$ with $\lambda_{\ell(\lambda)}\le m-1$ or $\lambda=\emptyset$, then $\phi$ is also an involution on $\Pi_m^*$. Since $(\emptyset,\emptyset)\not\in \Pi_m^*$, we have
	\begin{align*}
	&\sum_{(\lambda,\kappa)\in \Pi_m^*}w(\lambda,\kappa)x^{\ell(\lambda)+\ell(\kappa)}q^{|\lambda|+|\kappa|}\\
	&\qquad= \sum_{(\lambda,\kappa)\in \Pi_m^*}w(\phi(\lambda),\phi(\kappa))x^{\ell(\phi(\lambda))+\ell(\phi(\kappa))}q^{|\phi(\lambda)|+|\phi(\kappa)|}\\
	&\qquad= - \sum_{(\lambda,\kappa)\in \Pi_m^*}w(\lambda,\kappa)x^{\ell(\lambda)+\ell(\kappa)}q^{|\lambda|+|\kappa|},
	\end{align*}
	and therefore,
	\begin{align*}
	\sum_{(\lambda,\kappa)\in \Pi_m^*}w(\lambda,\kappa)x^{\ell(\lambda)+\ell(\kappa)}q^{|\lambda|+|\kappa|}= 0.
	\end{align*}
	Recall also that $\Pi_m\backslash \Pi_m^* = \sP_{g,\le m-1}^{(s)} \times \{\emptyset\}$. To see \eqref{eq:Gm}, 
	\begin{align*}
	P_g^{(s)}(-x,q)C_{g,\ge m}^{(s)}(x,q)&=\sum_{(\lambda,\kappa)\in \Pi_m}w(\lambda,\kappa)x^{\ell(\lambda)+\ell(\kappa)}q^{|\lambda|+|\kappa|}\\
	&=\sum_{(\lambda,\kappa)\in \Pi_m\backslash \Pi_m^*}w(\lambda,\kappa)x^{\ell(\lambda)+\ell(\kappa)}q^{|\lambda|+|\kappa|}\\
	&=\sum_{\lambda\in\sP_{g,\le m-1}^{(s)}}(-1)^{\ell(\lambda)}x^{\ell(\lambda)}q^{|\lambda|}\\
	&=P_{g,\le {m-1}}^{(s)}(-x,q).
	\end{align*}
	
	Finally, \eqref{eq:G} follows from \eqref{eq:Gm} since $C_g^{(s)}(x,q)=C_{g,\ge 1}^{(s)}(x,q)$ and $P_{g,\le 0}^{(s)}(-x,q)=1$.
\end{proof}

\section{Matrices}\label{sec:mat}

We first extend Theorem \ref{th:ABH1}. Let $K_g^{(s)}(n,m)$ denote the number of $m$-step compositions of $n$ in $\sC_g^{(s)}$.

\begin{theorem}
	For each positive integer $m$,
	\begin{equation}\label{eq:K-ind}
	\sum_{n\ge 0}K_g^{(s)}(n,m)q^n P_{g,\le n+m}^{(s)}(-1,q)=1.
	\end{equation}
\end{theorem}

\begin{proof}
	For fixed $m$, any composition $\kappa=(\kappa_1,\kappa_2,\ldots,\kappa_{\ell(\kappa)})$ in $\sC_g^{(s)}$ can be uniquely decomposed as
	\begin{equation*}
	\left\{
	\begin{aligned}
	\kappa'&=(\kappa_1,\kappa_2,\ldots,\kappa_{M}),\\
	\kappa''&=(\kappa_{M+1},\kappa_{M+2},\ldots,\kappa_{\ell(\kappa)}),
	\end{aligned}
	\right.
	\end{equation*}
	where $\kappa'\in \sC_g^{(s)}$ is $m$-step and $\kappa''\in\sC_g^{(s)}$ satisfies $\kappa_{M+1}\ge |\kappa'|+m+1$. Therefore,
	\begin{align*}
	\sum_{\kappa\in\sC_g^{(s)}}q^{|\kappa|}=\sum_{\substack{\kappa'\in\sC_g^{(s)}\\\text{$\kappa'$ is $m$-step}}}q^{|\kappa'|}\sum_{\substack{\kappa''\in\sC_g^{(s)}\\\kappa''_1\ge |\kappa'|+m+1}}q^{|\kappa''|}.
	\end{align*}
	Namely,
	\begin{align*}
	\frac{1}{P_g^{(s)}(-1,q)}=C_g^{(s)}(1,q)&=\sum_{n\ge 0}K_g^{(s)}(n,m)q^n C_{g,\ge n+m+1}^{(s)}(1,q)\\
	&=\sum_{n\ge 0}K_g^{(s)}(n,m)q^n \frac{P_{g,\le {m+n}}^{(s)}(-1,q)}{P_g^{(s)}(-1,q)},
	\end{align*}
	where we make use of Theorem \ref{th:G}. The desired result therefore follows.
\end{proof}

\begin{example}
	Taking $g=1$ and $s=1$ in \eqref{eq:Fm} gives
	$$P_{1,\le m}^{(1)}(-1,q)=1+\sum_{\ell\ge 1}(-1)^\ell q^{\ell+\binom{\ell}{2}}\qbinom{m}{\ell}_q=(q;q)_m,$$
	where we use Euler's second summation \cite[(3.3.6)]{And1976}. Recall that $K_1^{(1)}(n,m)$ counts \textit{the number of $m$-step partitions of $n$}. Then the $g=1$ and $s=1$ case of \eqref{eq:K-ind} yields
	\begin{equation}
	\sum_{n\ge 0}K_1^{(1)}(n,m)q^n (q;q)_{m+n}=1.
	\end{equation}
	This is Theorem \ref{th:ABH1}. Also, taking $g=0$ and $s=1$ in \eqref{eq:Fm} gives
	$$P_{0,\le m}^{(1)}(-1,q)=1+\sum_{\ell\ge 1}(-1)^\ell q^{\ell}\qbinom{m+\ell-1}{\ell}_q=\frac{1}{(-q;q)_m},$$
	where we use Euler's first summation \cite[(3.3.7)]{And1976}. Further, $K_0^{(1)}(n,m)$ counts \textit{the number of $m$-step partitions of $n$ into distinct parts}. Therefore, by \eqref{eq:K-ind},
	\begin{equation}
	\sum_{n\ge 0}\frac{K_0^{(1)}(n,m)q^n}{(-q;q)_{m+n}}=1.
	\end{equation}
\end{example}

Now we are ready to present a refinement of Theorem \ref{th:mu-gamma}. Let $M_g^{(s)}(n,m)$ be defined by
\begin{equation}
\sum_{m,n\ge 1}M_g^{(s)}(n,m)x^m q^n = \sum_{\lambda\in\sP_g^{(s)}\,(\lambda\ne\emptyset)} (-1)^{\ell(\lambda)+1}x^{\lambda_{\ell(\lambda)}}q^{|\lambda|}.
\end{equation}

\begin{theorem}\label{th:mu-gamma-ext}
	For positive integers $g$ and $s$, let $\mu_g^{(s)}$ be the infinite lower-triangular matrix defined by $(\mu_g^{(s)})_{i,j\ge 1}:=M_g^{(s)}(i+g+s-1,j+g+s-1)$ and let $\gamma_g^{(s)}$ be the infinite lower-triangular matrix defined by $(\gamma_g^{(s)})_{i,j\ge 1}:=K_g^{(s)}(i-j,j+s-1)$. Then $\mu_g^{(s)}$ and $\gamma_g^{(s)}$ are inverses of one another.
\end{theorem}

\begin{proof}
	For each positive integer $m$, we write
	$$P_{g,m}^{(s)}(q):=\sum_{\substack{\lambda\in\sP_g^{(s)}\,(\lambda\ne\emptyset)\\\lambda_{\ell(\lambda)}=m}}(-1)^{\ell(\lambda)+1}q^{|\lambda|}.$$
	Then
	$$P_{g,m}^{(s)}(q)=\begin{cases}
	0 & \text{if $1\le m\le s-1$},\\
	q^mP_{g,\le m-g}^{(s)}(-1,q) & \text{if $m\ge s$},
	\end{cases}$$
	where we put $P_{g,\le M}^{(s)}(-1,q)=1$ for $M<0$.
	
	Recall that $K_g^{(s)}(n,m)=0$ for $n<0$. We have, for $i,j\ge 1$,
	\begin{align*}
	(\mu_g^{(s)}\cdot \gamma_g^{(s)})_{i,j}&=\sum_{k\ge 1}(\mu_g^{(s)})_{i,k}(\gamma_g^{(s)})_{k,j}\\
	&=\sum_{k\ge j} M_g^{(s)}(i+g+s-1,k+g+s-1)K_g^{(s)}(k-j,j+s-1)\\
	&=\sum_{k\ge j} K_g^{(s)}(k-j,j+s-1)[q^{i+g+s-1}]P_{g,k+g+s-1}^{(s)}(q)\\
	&=[q^{i+g+s-1}]\sum_{k\ge j} K_g^{(s)}(k-j,j+s-1)q^{k+g+s-1}P_{g,\le k+s-1}^{(s)}(-1,q)\\
	&=[q^{i+g+s-1}]\sum_{k\ge 0} K_g^{(s)}(k,j+s-1)q^{k+j+g+s-1}P_{g,\le k+j+s-1}^{(s)}(-1,q)\\
	&=[q^i]q^j\sum_{k\ge 0} K_g^{(s)}(k,j+s-1)q^{k}P_{g,\le k+j+s-1}^{(s)}(-1,q)\\
	&=[q^i]q^j
	\end{align*}
	by applying \eqref{eq:K-ind}. We therefore conclude that $(\mu_g^{(s)}\cdot \gamma_g^{(s)})_{i,j}$ equals $1$ if $i=j$ and $0$ otherwise and thus arrive at the desired result.
\end{proof}

Below we provide an example of Theorem \ref{th:mu-gamma-ext}.

\begin{example}\label{ex:g=2}
	Let $g=2$ and $s=1$. Then
	\begin{equation*}
	\mu_2^{(1)}=\left(
	\begin{array}{ccccccccccccc}
	1 & 0 & 0 & 0 & 0 & 0 & 0 & 0 & 0 & 0 & 0 & 0 & \cdots\\
	-1 & 1 & 0 & 0 & 0 & 0 & 0 & 0 & 0 & 0 & 0 & 0 & \cdots\\
	0 & -1 & 1 & 0 & 0 & 0 & 0 & 0 & 0 & 0 & 0 & 0 & \cdots\\
	0 & -1 & -1 & 1 & 0 & 0 & 0 & 0 & 0 & 0 & 0 & 0 & \cdots\\
	0 & 0 & -1 & -1 & 1 & 0 & 0 & 0 & 0 & 0 & 0 & 0 & \cdots\\
	0 & 0 & -1 & -1 & -1 & 1 & 0 & 0 & 0 & 0 & 0 & 0 & \cdots\\
	0 & 0 & 1 & -1 & -1 & -1 & 1 & 0 & 0 & 0 & 0 & 0 & \cdots\\
	0 & 0 & 0 & 0 & -1 & -1 & -1 & 1 & 0 & 0 & 0 & 0 & \cdots\\
	0 & 0 & 0 & 1 & 0 & -1 & -1 & -1 & 1 & 0 & 0 & 0 & \cdots\\
	0 & 0 & 0 & 1 & 0 & 0 & -1 & -1 & -1 & 1 & 0 & 0 & \cdots\\
	0 & 0 & 0 & 0 & 2 & 0 & 0 & -1 & -1 & -1 & 1 & 0 & \cdots\\
	0 & 0 & 0 & 0 & 1 & 1 & 0 & 0 & -1 & -1 & -1 & 1 & \cdots\\
	\vdots & \vdots & \vdots & \vdots & \vdots & \vdots & \vdots & \vdots & \vdots & \vdots & \vdots & \vdots & \ddots\\
	\end{array}
	\right)
	\end{equation*}
	and
	\begin{equation*}
	\gamma_2^{(1)}=\left(
	\begin{array}{ccccccccccccc}
	1 & 0 & 0 & 0 & 0 & 0 & 0 & 0 & 0 & 0 & 0 & 0 & \cdots\\
	1 & 1 & 0 & 0 & 0 & 0 & 0 & 0 & 0 & 0 & 0 & 0 & \cdots\\
	1 & 1 & 1 & 0 & 0 & 0 & 0 & 0 & 0 & 0 & 0 & 0 & \cdots\\
	2 & 2 & 1 & 1 & 0 & 0 & 0 & 0 & 0 & 0 & 0 & 0 & \cdots\\
	3 & 3 & 2 & 1 & 1 & 0 & 0 & 0 & 0 & 0 & 0 & 0 & \cdots\\
	6 & 6 & 4 & 2 & 1 & 1 & 0 & 0 & 0 & 0 & 0 & 0 & \cdots\\
	10 & 10 & 6 & 4 & 2 & 1 & 1 & 0 & 0 & 0 & 0 & 0 & \cdots\\
	19 & 19 & 12 & 7 & 4 & 2 & 1 & 1 & 0 & 0 & 0 & 0 & \cdots\\
	33 & 33 & 21 & 12 & 7 & 4 & 2 & 1 & 1 & 0 & 0 & 0 & \cdots\\
	60 & 60 & 38 & 22 & 13 & 7 & 4 & 2 & 1 & 1 & 0 & 0 & \cdots\\
	106 & 106 & 67 & 39 & 22 & 13 & 7 & 4 & 2 & 1 & 1 & 0 & \cdots\\
	190 & 190 & 120 & 70 & 40 & 23 & 13 & 7 & 4 & 2 & 1 & 1 & \cdots\\
	\vdots & \vdots & \vdots & \vdots & \vdots & \vdots & \vdots & \vdots & \vdots & \vdots & \vdots & \vdots & \ddots\\
	\end{array}
	\right).
	\end{equation*}
\end{example}

We close this section with another observation about the matrix $\gamma_g^{(s)}$.

\begin{proposition}\label{prop:gamma}
	Let $k$ be a nonnegative integer and $n\ge 2k-s+1$ be a positive integer. For each nonnegative integer $g$ and positive integer $s$, we have that $(\gamma_g^{(s)})_{n,n-k}$ equals the number of compositions of $k$ in $\sC_g^{(s)}$.
	
	More generally, for nonnegative integers $g_1,\ldots,g_M$, if $c_{g_1,\ldots,g_M}^{(s)}(k)$ counts the number of $M$-tuples of compositions $(\kappa^{(1)},\ldots,\kappa^{(M)})$ such that $\kappa^{(j)}\in\sC_{g_j}^{(s)}$ for $1\le j\le M$ and $|\kappa^{(1)}|+\cdots+|\kappa^{(M)}|=k$, then $(\gamma_{g_1}^{(s)}\cdots \gamma_{g_M}^{(s)})_{n,n-k}=c_{g_1,\ldots,g_M}^{(s)}(k)$.
\end{proposition}

\begin{proof}
	For the first part, $(\gamma_g^{(s)})_{n,n-k}=K_g^{(s)}(k,n-k+s-1)$ is the number of $(n-k+s-1)$-step compositions of $k$ in $\sC_g^{(s)}$. However, since $n-k+s-1\ge k$, every composition of $k$ is $(n-k+s-1)$-step.
	
	We then conclude the second part by induction on the number matrices in the product as soon as we notice that this product is always lower-triangular. If this statement is true for some $g_1,\ldots,g_M$, then
	\begin{align*}
	(\gamma_{g_1}^{(s)}\cdots \gamma_{g_M}^{(s)}\cdot \gamma_{g_{M+1}}^{(s)})_{n,n-k}&=\sum_{i\ge 1} (\gamma_{g_1}^{(s)}\cdots \gamma_{g_M}^{(s)})_{n,i}(\gamma_{g_{M+1}}^{(s)})_{i,n-k}\\
	&=\sum_{i=n-k}^{n} (\gamma_{g_1}^{(s)}\cdots \gamma_{g_M}^{(s)})_{n,i}(\gamma_{g_{M+1}}^{(s)})_{i,n-k}\\
	&=\sum_{i=n-k}^{n} c_{g_1,\ldots,g_M}^{(s)}(n-i)\cdot c_{g_{M+1}}^{(s)}(i-n+k)\\
	&=c_{g_1,\ldots,g_{M+1}}^{(s)}(k)
	\end{align*}
	for any nonnegative $g_{M+1}$.
\end{proof}

\begin{example}
	Let $p(n)$ , $p_D(n)$, and $\overline{p}(n)$ count the number of unrestricted partitions, distinct partitions, and overpartitions of $n$, respectively. Then for each nonnegative integer $k$ and positive integer $n\ge 2k$, $(\gamma_1^{(1)})_{n,n-k}=p(k)$, $(\gamma_0^{(1)})_{n,n-k}=p_D(k)$ and $(\gamma_0^{(1)}\cdot \gamma_1^{(1)})_{n,n-k}=(\gamma_1^{(1)}\cdot \gamma_0^{(1)})_{n,n-k}=\overline{p}(k)$.
\end{example}

\subsection*{Acknowledgements}

The second author was supported by a Killam Postdoctoral Fellowship from the Killam Trusts. We would like to thank George Andrews and Karl Dilcher for many helpful comments.

\bibliographystyle{amsplain}

\end{document}